\documentclass{amsart}

\newcommand{\SR}{{\mathcal R}}
\newcommand{\SU}{{\mathcal U}}
\newcommand{\di}{\operatorname{dim}}
\usepackage{amsmath,amssymb, amstext,amsthm,amsfonts}


\title[A dichotomy for higher dimensional flows]
      {A dichotomy for higher dimensional flows}

\author[A. Arbieto, C. A. Morales]{A. Arbieto, C. A. Morales}
\address{Instituto de Mat\'ematica, UFRJ,
P. O. Box 68530, 21945-970 Rio de Janeiro, Brazil.}
\email{arbieto@im.ufrj.br, morales@impa.br}

\thanks{Partially supported by CNPq, CAPES-Prodoc, FAPERJ and PRONEX/DS from Brazil.}

\newtheorem{theorem}{Theorem}
\newtheorem{corollary}[theorem]{Corollary}
\newtheorem*{main1}{Theorem A}
\newtheorem*{main2}{Theorem B}
\newtheorem{lemma}[theorem]{Lemma}
\newtheorem{proposition}[theorem]{Proposition}
\newtheorem{conjecture}[theorem]{Conjecture}

\newtheorem{definition}[theorem]{Definition}

\begin{document}

\begin{abstract}
We analyze the dichotomy between {\em sectional-Axiom A flows} (c.f. \cite{memo})
and flows with points accumulated by periodic orbits of different indices.
Indeed, this is proved for $C^1$ generic flows whose singularities accumulated by
periodic orbits have codimension one.
Our result improves \cite{mp1}.
\end{abstract}

\maketitle

\section{Introduction}

\noindent
Ma\~n\'e discussed
in his breakthrough work \cite{M} if
the {\em star property}, i.e., the property of being far away from systems with non-hyperbolic
periodic orbits, is sufficient to guarantee that a system be Axiom A.
Although this is true for diffeomorphisms \cite{h0} it is not
for flows as the geometric
Lorenz attractor \cite{abs}, \cite{gu}, \cite{GW} shows. On the other hand, if singularities are not allowed
then the answer turns on to be positive by \cite{gw}. Previously, Ma\~n\'e
connects the star property with
the nowadays called {\em Newhouse phenomenon} at least for surfaces. In fact, he proved that a $C^1$-generic
surface diffeomorphism either is Axiom A or displays infinitely many sinks or sources \cite{m}.
In the extension of this work on surfaces,
\cite{mp1} obtained
the following results about $C^1$-generic flows for closed 3-manifolds: Any $C^1$-generic
star flow is singular-Axiom A and, consequently,
any $C^1$-generic flow is singular-Axiom A or displays infinitely many
sinks or sources. The notion of {\em singular-Axiom A} was introduced in \cite{mpp}
inspired on the dynamical properties of both Axiom A flows and the
geometric Lorenz attractor.
It is then natural to investigate such generic phenomena in higher dimensions and the natural challenges are:
Is a $C^1$-generic star flow in a closed $n$-manifold
singular-Axiom A?
Does a $C^1$-generic vector field in a closed $n$-manifold is singular-Axiom A or has
infinitely many sinks or sources?
Unfortunately, what we know is that the second question has negative answer for $n\geq 5$
as counterexamples
can be obtained by suspending the diffeomorphisms in Theorem C of \cite{bv}
(but for $n=4$ the answer may be positive).
A new light
comes from the {\em sectional-Axiom A flows} introduced in \cite{memo}.
Indeed, the first author replaced the term singular-Axiom A by sectional-Axiom A
above in order to formulated the following conjecture
for $n\geq 3$ (improving that in p. 947 of \cite{gwz}):

\begin{conjecture}
 \label{conj0}
$C^1$-generic star flows on closed $n$-manifolds are sectional-Axiom A.
\end{conjecture}

Analogously we can ask if a
$C^1$-generic vector field in a closed $n$-manifold is sectional-Axiom A or display
infinitely many sinks or sources.
But now the answer is
negative not only for $n=5$, by the suspension of \cite{bv} as above,
but also for $n=4$ by \cite{st} and the suspension of certain diffeomorphisms \cite{mane}.
Nevertheless, in all these counterexamples, it is possible to observe the existence of
{\em points accumulated by hyperbolic periodic orbits of different Morse indices}.
Since such a phenomenon can be observe also
in a number of well-known examples of non-hyperbolic systems and
since, in dimension three, that phenomenon implies
existence of infinitely many sinks or sources,
it is possible to formulate the following dichotomy
(which, in virtue of Proposition \ref{p1}, follows from Conjecture \ref{conj0}):

\begin{conjecture}
\label{conj1}
$C^1$-generic vector fields $X$
satisfy (only) one of the following
properties:
\begin{enumerate}
 \item
$X$ has a point accumulated by hyperbolic periodic orbits of different Morse indices;
\item
$X$ is sectional-Axiom A.
\end{enumerate}
\end{conjecture}

In this paper we prove Conjecture \ref{conj1} but in a case very close to the
three-dimensional one, namely,
when the {\em singularities accumulated by
periodic orbits have codimension one} (i.e. Morse index $1$ or $n-1$).
Observe that our result
implies the dichotomy in \cite{mp1}
since the assumption about the singularities is automatic for $n=3$.
It also implies Conjecture \ref{conj1} in large classes of vector fields as, for instance,
those whose singularities (if any) have codimension one.
As an application we prove Conjecture \ref{conj0} for star flows with spectral decomposition
as soon as the singularities accumulated by periodic orbits have codimension one.
Let us state our results in a precise way.

In what follows $M$ is a compact connected boundaryless Riemannian manifold of dimension $n\geq 3$
(a {\em closed $n$-manifold} for short).
If $X$ is a $C^1$ vector field in $M$ we will
denote by $X_t$ the flow generated by $X$ in $M$.
A subset $\Lambda\subset M$ is
{\em invariant} if
$X_t(\Lambda)=\Lambda$ for all $t\in I \!\! R$.
By a {\em closed orbit} we mean a periodic orbit or a singularity.
We define the {\em omega-limit set} of $p\in M$ by
$$
\omega(p)=\left\{x\in M:
x=\lim_{n\to\infty}X_{t_n}(p)
\mbox{ for some sequence }t_n\to\infty\right\}
$$
and call $\Lambda$
{\em transitive} if
$\Lambda=\omega(p)$ for some $p\in \Lambda$.
Clearly every transitive set is compact invariant.
As customary we call $\Lambda$ {\em nontrivial} it it does not reduce to a single orbit.

Denote by $\|\cdot\|$ and $m(\cdot)$ the norm and the minimal norm
induced by the Riemannian metric and by
$Det(\cdot)$ the jacobian operation.
A compact invariant set $\Lambda$ is {\em hyperbolic}
if there are a continuous invariant tangent bundle decomposition
$$
T_\Lambda M=\hat{E}^s_\Lambda\oplus E^X_\Lambda\oplus \hat{E}^u_\Lambda
$$
and positive constants $K,\lambda$
such that $E^X_\Lambda$ is the subbundle generated by $X$,
$$
\|DX_t(x)/\hat{E}^s_x\|\leq Ke^{-\lambda t}
\quad \mbox{ and }\quad m(DX_t(x)/\hat{E}^u_x)\geq K^{-1}e^{\lambda t},
$$
for all $x\in \Lambda$ and $t\geq 0$.
Sometimes we write $\hat{E}^{s,X}_x$, $\hat{E}^{u,X}_x$ to indicate dependence on $X$.

A closed orbit $O$ is hyperbolic if it does as a compact invariant set. In such a case
we define its {\em Morse index} $I(O)=dim(\hat{E}^s_O)$, where $dim(\cdot)$
stands for the dimension operation.
If $O$ reduces to a singularity $\sigma$, then we write
$I(\sigma)$ instead of $I(\{\sigma\})$ and say
that $\sigma$ has {\em codimension one} if $I(\sigma)=1$ or $I(\sigma)=n-1$.
It is customary to call hyperbolic closed orbit of maximal (resp. minimal) Morse index
{\em sink} (resp. {\em source}).

On the other hand, an invariant splitting
$T_\Lambda M=E_\Lambda\oplus F_\Lambda$
over $\Lambda$ is {\em dominated}
(we also say that $E_\Lambda$ {\em dominates} $F_\Lambda$) if there are positive constants
$K,\lambda$ such that
$$
\frac{\|DX_t(x)/E_x\|}{m(DX_t(x)/F_x)}\leq Ke^{-\lambda t},
\quad\quad\forall x\in \Lambda \mbox{ and }t\geq 0.
$$

In this work we agree to call a compact invariant set $\Lambda$
{\em partially hyperbolic} if there is a dominated splitting $T_\Lambda M=E^s_\Lambda\oplus E^c_\Lambda$
with {\em contracting} dominating subbundle $E^s_\Lambda$,
namely,
$$
\|DX_t(x)/E^s_x\|\leq Ke^{-\lambda t},
\quad\quad\forall x\in \Lambda \mbox{ and }t\geq 0.
$$
We stress however that this is not a standard usage
(specially due to the lack of symmetry in this definition).
Anyway, in such a case, we say that $\Lambda$ has {\em contracting dimension $d$}
if $dim(E^s_x)=d$ for all $x\in \Lambda$.
Moreover, we say that the central subbundle $E^c_\Lambda$ is {\em sectionally expanding} if
$$
dim(E^c_x)\geq 2 \quad\mbox{ and }\quad
|Det(DX_t(x)/L_x)|\geq K^{-1}e^{\lambda t},
\quad\quad\forall x\in \Lambda \mbox{ and }t\geq 0
$$
and all two-dimensional subspace
$L_x$ of $E^c_x$.

A {\em sectional-hyperbolic set} is a partially hyperbolic set whose singularities (if any) are hyperbolic
and whose central subbundle is sectionally expanding
(\footnote{Some authors use the term {\em singular-hyperbolic} instead.}).

Now we recall the concept of sectional-Axiom A flow \cite{memo}.
Call a point $p\in M$ {\em nonwandering} if for every neighborhood
$U$ of $p$ and every $T>0$ there is
$t>T$ such that $X_t(U)\cap U\neq\emptyset$.
We denote by $\Omega(X)$ the set of nonwandering points of $X$ (which
is clearly a compact invariant set).
We say that $X$ is an {\em Axiom A flow} if $\Omega(X)$ is both hyperbolic
and the closure of the closed orbits.
The so-called
{\em Spectral Decomposition Theorem} \cite{hk}
asserts that
the nonwandering set of an Axiom A flow $X$
splits into finitely many disjoint
transitive sets {\em with dense closed orbits} (i.e. with a dense subset of closed orbits)
which are hyperbolic for $X$.
This motivates the following definition:

\begin{definition}
A $C^1$ vector field $X$ in $M$ is called
{\em sectional-Axiom A flow} if there is a finite disjoint decomposition
$
\Omega(X)=\Omega_1\cup \cdots \cup \Omega_k
$
formed by transitive sets with dense periodic orbits
$\Omega_1,\cdots, \Omega_k$ such that, for all $1\leq i\leq k$,
$\Omega_i$ is either a hyperbolic set for $X$ or a sectional-hyperbolic set for $X$ or a sectional-hyperbolic
set for $-X$.
\end{definition}

Let $\mathcal{X}^1$ denote the space of $C^1$ vector fields $X$ in $M$.
Notice that it is a Baire space if equipped with the standard $C^1$ topology.
The expression {\em $C^1$-generic vector field} will mean a vector field in a
certain residual subset of $\mathcal{X}^1$.
We say that a point is {\em accumulated by periodic orbits},
if it lies in the closure of the union of the periodic orbits, and
{\em accumulated by hyperbolic periodic orbits of different Morse index}
if it lies simultaneously in the closure of the hyperbolic periodic orbits of Morse index
$i$ and $j$ with $i\neq j$.
With these definitions we can state our main result settling a special case of Conjecture \ref{conj1}.

\begin{main1}
A $C^1$-generic vector field $X$ for which
the singularities accumulated by periodic orbits
have codimension one satisfies (only) one of the following
properties:
\begin{enumerate}
 \item
$X$ has a point accumulated by hyperbolic periodic orbits of different Morse indices;
\item
$X$ is sectional-Axiom A.
\end{enumerate}
\end{main1}

Standard $C^1$-generic results \cite{cmp} imply
that the sectional-Axiom A flows in the second alternative above
also satisfy the no-cycle condition.

The proof of our result follows that of Theorem A in \cite{mp1}.
However, we need a more direct approach bypassing Conjecture \ref{conj0}.
Indeed, we shall use some methods in \cite{mp1}
together with a combination of results \cite{glw}, \cite{gwz}, \cite{memo} for nontrivial transitive sets
(originally proved for robustly transitive sets).

\begin{definition}[\cite{a}]
We say that $X$ has
{\em spectral decomposition} if there is a finite partition $\Omega(X)=\Lambda_1\cup\cdots\cup\Lambda_l$ formed
by transitive sets $\Lambda_1,\cdots, \Lambda_l$ .
\end{definition}

Theorem A will imply the following approach to Conjecture \ref{conj0}.

\begin{main2}
\label{the-coro}
A $C^1$-generic star flow with spectral decomposition and
for which the singularities accumulated by periodic orbits have codimension one
is sectional-Axiom A.
\end{main2}

\section{Proofs}
\label{sec2}

\noindent
Hereafter we fix a closed $n$-manifold $M$, $n\geq 3$,
$X\in \mathcal{X}^1$ and a
compact invariant set $\Lambda$ of $X$.
Denote by $Sing(X,\Lambda)$ the set of singularities of $X$ in $\Lambda$.
We shall use the following concept from \cite{glw}.

\begin{definition}
We say that $\Lambda$
has a definite index $0\leq Ind(\Lambda)\leq n-1$ if there are
a neighborhood $\mathcal{U}$ of $X$ in $\mathcal{X}^1$ and
a neighborhood $U$ of $\Lambda$ in $M$ such that
$I(O)=Ind(\Lambda)$ for
every hyperbolic periodic orbit $O\subset U$ of every vector field $Y\in \mathcal{U}$.
In such a case we say that $\Lambda$ is {\em strongly homogeneous (of index $Ind(\Lambda)$)}.
\end{definition}

It turns out that the strongly homogeneous property imposes certain constraints on the Morse indices of
the singularities \cite{gwz}. To explain this we use the concept of
{\em saddle value} of a hyperbolic singularity $\sigma$ of $X$ defined by
$$
\Delta(\sigma)=Re(\lambda)+Re(\gamma)
$$
where $\lambda$ (resp. $\gamma$) is the stable (resp. unstable) eigenvalue
with maximal (resp. minimal) real part
(c.f. \cite{sstc} p. 725).
Indeed, based on the Hayashi's connecting lemma \cite{h} and well-known results about
unfolding of homoclinic loops \cite{sstc}, Lemma 4.3 in \cite{gwz} proves that,
if $\Lambda$ is a robustly transitive set which is strongly
homogeneous with hyperbolic singularities,
then $\Delta(\sigma)\neq 0$
and, furthermore, $I(\sigma)=Ind(\Lambda)$ or $Ind(\Lambda)+1$ depending on whether
$\Delta(\sigma)<0$ or $\Delta(\sigma)>0$, $\forall \sigma\in Sing(X,\Lambda)$.
However, we can observe that the same is true for
nontrivial transitive sets (instead of robustly transitive sets) for the proof in
\cite{gwz} uses the connecting lemma only once.
In this way we obtain the following lemma.

\begin{lemma}
\label{43}
Let $\Lambda$ be a nontrivial transitive set which is strongly
homogeneous with singularities (all hyperbolic) of $X$.
Then, every $\sigma\in Sing(X,\Lambda)$ satisfies $\Delta(\sigma)\neq 0$
and one of the properties below:
\begin{itemize}
\item
If $\Delta(\sigma)<0$, then $I(\sigma)=Ind(\Lambda)$.
\item
If $\Delta(\sigma)>0$, then $I(\sigma)=Ind(\Lambda)+1$.
\end{itemize}
\end{lemma}

On the other hand, the following inequalities for strongly homogeneous sets $\Lambda$
where introduced in \cite{glw}:
\begin{equation}
\label{eq1}
I(\sigma)>Ind(\Lambda),
\quad\quad\forall \sigma\in Sing(X,\Lambda).
\end{equation}
\begin{equation}
 \label{eq11}
I(\sigma)\leq Ind(\Lambda),
\quad\quad\forall \sigma\in Sing(X,\Lambda).
\end{equation}

We shall use the above lemma to present a special case where one of these
inequalities can be proved.

\begin{proposition}
\label{thCcc}
Let $\Lambda$ be a nontrivial transitive set which is strongly
homogeneous with singularities (all hyperbolic of codimension one) of $X$.
If $n\geq 4$ and $1\leq Ind(\Lambda)\leq n-2$, then $\Lambda$
satisfies either (\ref{eq1}) or (\ref{eq11}).
\end{proposition}

\begin{proof}
Otherwise
there are $\sigma_0,\sigma_1\in Sing(X,\Lambda)$
satisfying
$I(\sigma_0)\leq Ind(\Lambda)<I(\sigma_1)$.
Since both $\sigma_0$ and $\sigma_1$ have codimension one and $1\leq Ind(\Lambda)\leq n-2$
we obtain $I(\sigma_0)=1$ and $I(\sigma_1)=n-1$.
If $\Delta(\sigma_0)\geq 0$ then $I(\sigma_0)=Ind(\Lambda)+1$
by Lemma \ref{43} so $Ind(\Lambda)=0$ which contradicts $1\leq Ind(\Lambda)$.
Then $\Delta(\sigma_0)<0$ and so $Ind(\Lambda)=I(\sigma_0)=1$ by Lemma \ref{43}.
On the other hand, if
$\Delta(\sigma_1)<0$ then
$Ind(\Lambda)=I(\sigma_1)=n-1$ by Lemma \ref{43}.
As $Ind(\Lambda)=1$ we get $n=2$ contradicting $n\geq 4$.
Then $\Delta(\sigma_1)\geq0$ so $I(\sigma_1)=Ind(\Lambda)+1$ by Lemma \ref{43}
thus $n=3$ contradicting $n\geq 4$.
The proof follows.
\end{proof}

The importance of (\ref{eq1}) and (\ref{eq11}) relies on the
the following result proved in \cite{glw}, \cite{gwz}, \cite{memo}:
A $C^1$ robustly transitive set $\Lambda$ with singularities
(all hyperbolic) which is strongly homogeneous satisfying (\ref{eq1}) (resp. (\ref{eq11}))
is sectional hyperbolic for $X$ (resp. $-X$).
However, we can observe that the same is true for
nontrivial transitive sets (instead of robustly transitive sets) as soon as $1\leq Ind(\Lambda)\leq n-2$.
The proof is similar to that in \cite{glw},\cite{gwz}, \cite{memo}
but using the so-called {\em preperiodic set} \cite{w}
instead of the natural continuation of a robustly transitive sets.
Combining this with Proposition \ref{thCc}
we obtain the following corollary in which the expression
{\em up to flow-reversing} means either for $X$ or $-X$.

\begin{corollary}
\label{thCc}
Let $\Lambda$ be a nontrivial transitive set which is strongly
homogeneous with singularities (all hyperbolic of codimension one) of $X$.
If $n\geq 4$ and $1\leq Ind(\Lambda)\leq n-2$, then $\Lambda$
is sectional-hyperbolic up to flow-reversing.
\end{corollary}

A direct application of this corollary is as follows.
We say that $\Lambda$
is {\em Lyapunov stable} for $X$
if for every neighborhood $U$ of it there is a
neighborhood $W\subset U$ of it such that $X_t(p)\in U$ for every $t\ge0$ and $p\in W$.

It was proved in Theorem C of \cite{mp1} that,
for $C^1$ generic three-dimensional star flows, every nontrivial Lyapunov stable set
with singularities is singular-hyperbolic.
We will need a similar result
for higher dimensional flows, but with the term singular-hyperbolic
replaced by sectional-hyperbolic.
The following will supply such a result.

\begin{corollary}
 \label{thC}
Let $\Lambda$ be a nontrivial transitive set which is strongly
homogeneous with singularities (all hyperbolic of codimension one) of $X$.
If $n\geq 4$, $1\leq Ind(\Lambda)\leq n-2$ and $\Lambda$ is Lyapunov stable, then $\Lambda$
is sectional-hyperbolic for $X$.
\end{corollary}

\begin{proof}
By Corollary \ref{thCc} it suffices to prove that $\Lambda$ cannot be sectional-hyperbolic for $-X$.
Assume by contradiction that it does. Then,
by integrating the corresponding contracting subbundle,
we obtain a strong stable manifold $W^{ss}_{-X}(x)$, $\forall x\in \Lambda$.
But $\Lambda$ is Lyapunov stable for $X$ so $W^{ss}_{-X}(x)\subset \Lambda$,
$\forall x\in \Lambda$, contradicting p. 556 in \cite{momo}.
Then, $\Lambda$ cannot be sectional-hyperbolic for $-X$ and we are done.
\end{proof}

We also use Lemma \ref{43} to prove the following proposition.

\begin{proposition}
\label{c1}
Every nontrivial transitive sectional-hyperbolic set $\Lambda$
of a vector field $X$ in a closed $n$-manifold, $n\geq 3$, is strongly homogeneous
and satisfies $I(\sigma)=Ind(\Lambda)+1$, $\forall \sigma\in Sing(X,\Lambda)$.
\end{proposition}

\begin{proof}
Since transitiveness implies connectedness we have that the strong stable subbundle
$E^s_\Lambda$ of $\Lambda$ has constant dimension.
From this and the persistence of
the sectional-hyperbolic splitting
we obtain that $\Lambda$ is strongly homogeneous of index $Ind(\Lambda)=dim(E^s_x)$,
for $x\in \Lambda$.
Now fix a singularity $\sigma$. To prove $I(\sigma)=Ind(\Lambda)+1$ we only need to prove
that $\Delta(\sigma)>0$ (c.f. Lemma \ref{43}).

Suppose by contradiction that $\Delta(\sigma)\leq 0$.
Then, $\Delta(\sigma)<0$ and $I(\sigma)=Ind(\Lambda)$ by Lemma \ref{43}.
Therefore, $dim(E^s_\sigma)=dim(\hat{E}^s_\sigma)$ where $T_\sigma M=\hat{E}^s_\sigma\oplus \hat{E}^u_\sigma$
is the hyperbolic splitting of $\sigma$ (as hyperbolic singularity of $X$).
Now, let $W^s(\sigma)$ the stable manifold of $\sigma$ and
$W^{ss}(\sigma)$ be the strong stable manifold of $\sigma$ obtained by integrating
the strong stable subbundle $E^s_\Lambda$ (c.f. \cite{hps}).
Notice that $W^{ss}(\sigma)\subset W^s(\sigma)$.
As
$dim(W^{ss}(\sigma))=dim(E^s_\sigma)=dim(\hat{E}_\sigma^s)=dim(W^s(\sigma))$ we get
$W^{ss}(\sigma)=W^s(\sigma)$.
But $\Lambda$ is nontrivial transitive so the dense orbit
will accumulate at some point in $W^s(\sigma)\setminus \{\sigma\}$.
As $W^{ss}(\sigma)=W^{s}(\sigma)$ such a point
must belong to $(\Lambda\cap W^{ss}(\sigma))\setminus \{\sigma\}$.
On the other hand, it is well known that $\Lambda\cap W^{ss}(\sigma)=\{\sigma\}$ (c.f. \cite{mp1})
so we obtain a contradiction which proves the result.
\end{proof}

We say that
$\Lambda$ is an {\em attracting set}
if there is a neighborhood $U$ of it such that
$$
\Lambda=\bigcap_{t>0}X_t(U).
$$
On the other hand, a {\em sectional-hyperbolic attractor}
is a transitive attracting set which is also a sectional-hyperbolic set.
An {\em unstable branch} of a hyperbolic singularity $\sigma$ of a vector field
is an orbit in $W^u(\sigma)\setminus\{\sigma\}$.
We say that $\Lambda$ has
{\em dense singular unstable branches} if every unstable branch of every hyperbolic singularity on it
is dense in $\Lambda$.

The following is a straightforward extension of Theorem D in \cite{mp1} to higher dimensions
(with similar proof).

\begin{proposition}
\label{thD}
Let $\Lambda$ be a Lyapunov stable sectional-hyperbolic set
of a vector field $X$ in a closed $n$-manifold, $n\geq 3$.
If $\Lambda$ has both singularities, all of Morse index $n-1$, and
dense singular unstable branches, then $\Lambda$ is a sectional-hyperbolic attractor of $X$.
\end{proposition}

Now we recall the star flow's terminology from \cite{w}.

\begin{definition}
\label{star-flow}
A {\em star flow} is a $C^1$ vector field which cannot be $C^1$-approximated
by ones exhibiting non-hyperbolic closed orbits.
\end{definition}

Corollary \ref{thC} together with propositions \ref{c1} and \ref{thD} implies
the key result below.

\begin{proposition}
\label{p1}
A $C^1$-generic vector field $X$ on a closed $n$-manifold, $\forall n\geq 3$,
without points accumulated by hyperbolic periodic orbits of different Morse indices
is a star flow. If, in addition, $n\geq 4$, then
the codimension one singularities of $X$ accumulated by periodic orbits
belong to a sectional-hyperbolic attractor up to flow-reversing.
\end{proposition}

\begin{proof}
We will use the following notation.
Given $Z\in \mathcal{X}^1$ and $0\leq i\leq n-1$ we denote by
$Per_i(Z)$ the union of the hyperbolic periodic orbits of Morse index $i$.
The closure operation will be denoted by $Cl(\cdot)$.

Since $X$ has no point accumulated by hyperbolic periodic orbits of different Morse indices
one has
\begin{equation}
 \label{separa}
Cl(Per_i(X))\cap Cl(Per_j(X))=\emptyset,
\quad\quad
\forall i,j\in \{0,\cdots, n-1\}, \quad i\neq j.
\end{equation}
Then, since $X$ is $C^1$-generic, standard lower-semicontinuous arguments (c.f. \cite{cmp})
imply that there are
a neighborhood $\mathcal{U}$ of $X$ in $\mathcal{X}^1$ and
a pairwise disjoint collection of neighborhoods $\{U_i: 0\leq i\leq n-1\}$ such that
$Cl(Per_i(Y))\subset U_i$ for all $0\leq i\leq n-1$ and $Y\in \mathcal{U}$.

Let us prove that $X$ is a star flow.
When necessary we use the notation $I_X(O)$ to indicate dependence on $X$.
By contradiction assume that $X$ is not a star flow.
Then, there is a vector field $Y\in \mathcal{U}$
exhibiting a non-hyperbolic closed orbit $O$.
Since $X$ is generic we can assume by the Kupka-Smale Theorem \cite{hk} that
$O$ is a periodic orbit.
Unfolding the eigenvalues of $O$ is a suitable way we would
obtain two vector fields $Z_1,Z_2\in \mathcal{U}$ of which
$O$ is a hyperbolic periodic orbit with $I_{Z_1}(O)\neq I_{Z_2}(O)$,
$1\leq I_{Z_1}(O)\leq n-1$ and $1\leq I_{Z_2}(O)\leq n-1$.
Consequently, $O\subset U_i\cap U_j$ where $i=I_{Z_1}(O)$ and $j=I_{Z_2}(O)$
which contradicts that the collection $\{U_i: 0\leq i\leq n-1\}$ is pairwise disjoint.
Therefore, $X$ is a star flow.

Next we prove that $Cl(Per_i(X))$ is a strongly homogeneous set of index $i$, $\forall 0\leq i\leq n-1$.
Take $Y\in \mathcal{U}$ and a hyperbolic periodic orbit
$O\subset U_i$ of Morse index $I_Y(O)=j$. Then, $O\subset Cl(Per_j(Y))$ and so
$O\subset U_j$ from which we get $O\subset U_i\cap U_j$.
As the collection $\{U_i: 0\leq i\leq n-1\}$ is disjoint we conclude that
$i=j$ and so
every hyperbolic periodic orbit $O\subset U_i$ of every vector field $Y\in \mathcal{U}$
has Morse index $I_Y(O)=i$.
Therefore, $Cl(Per_i(X))$ is a strongly homogeneous set of index $i$.

Now, we prove that every codimension one singularity $\sigma$
accumulated by periodic orbits belongs to a sectional-hyperbolic attractor
up to flow-reversing.
More precisely, we prove that if $I(\sigma)=n-1$ (resp. $I(\sigma)=1$), then $\sigma$ belongs to a
sectional-hyperbolic attractor of $X$ (resp. of $-X$).
We only consider the case $I(\sigma)=n-1$ for the case $I(\sigma)=1$
can be handled analogously by just replacing $X$ by $-X$.

Since $I(\sigma)=n-1$ one has $dim(W^u(\sigma))=1$ and,
since $X$ is generic, we can assume that both
$Cl(W^u(\sigma))$ and $\omega(q)$ (for $q\in W^u(\sigma)\setminus\{\sigma\}$) are Lyapunov stable sets of
$X$ (c.f. \cite{cmp'}). As $\sigma$ is accumulated by periodic orbits we obtain from Lemma 4.2 in \cite{mp1} that
$Cl(W^u(\sigma))$ is a transitive set.

We claim that $Cl(W^u(\sigma))$ is strongly homogeneous.
Indeed, since $X$ is generic the General Density Theorem \cite{p} implies $\Omega(X)=Cl(Per(X)\cup Sing(X))$.
Denote by $Sing^*(X)$ is the set of singularities accumulated by periodic orbits.
Then, there is a decomposition
$$
\Omega(X)=\left(\bigcup_{0\leq i\leq n-1}
Cl(Per_i(X))\right)\cup\left(\bigcup_{\sigma'\in Sing(X)\setminus Sing^*(X)}\{\sigma'\}\right)
$$
which is disjoint by (\ref{separa}).
In addition,
$Cl(W^u(\sigma))$ is transitive and so it is connected
and contained in $\Omega(X)$.
As $\sigma\in Sing^*(X)$ by hypothesis
we conclude that
$Cl(W^u(\sigma)) \subset Cl(Per_{i_0}(X))$ for some $0\leq i_0\leq n-1$.
But we have proved above that $Cl(Per_{i_0}(X))$ is a strongly homogeneous set of index $i_0$,
so, $Cl(W^u(\sigma))$ is also a
strongly homogeneous set of index $i_0$. The claim follows.

On the other hand,
$X$ is a star flow and so it has finitely many sinks and sources \cite{li}, \cite{pl}.
From this we obtain $1\leq i_0\leq n-2$ and so $1\leq Ind(Cl(W^u(\sigma)))\leq n-2$.
Summarizing, we have proved that $Cl(W^u(\sigma))$ is a transitive set with singularities,
all of them of codimension one, which is a Lyapunov stable strongly homogeneous set
of index $1\leq Ind(Cl(W^u(\sigma)))\leq n-2$.
As certainly $Cl(W^u(\sigma))$ is nontrivial Corollary \ref{thC}
applied to $\Lambda=Cl(W^u(\sigma))$ implies that $Cl(W^u(\sigma))$ is sectional-hyperbolic.

Once we have proved that $Cl(W^u(\sigma))$ is sectional-hyperbolic
we apply Proposition \ref{c1} to $\Lambda=Cl(W^u(\sigma))$ yielding
$I(\sigma')=i_0+1$, $\forall\sigma'\in Sing(X,Cl(W^u(\sigma)))$.
But $\sigma\in Cl(W^u(\sigma))$ and $I(\sigma)=n-1$ so
$i_0=n-2$ by taking $\sigma'=\sigma$ above. Consequently, $I(\sigma')=n-1$ and so
$dim(W^u(\sigma'))=1$, $\forall\sigma'\in Cl(W^u(\sigma))$.
This implies two things. Firstly that every singularity in $Cl(W^u(\sigma))$ has Morse index $n-1$ and,
secondly, since $X$ is generic, we can assume that
$Cl(W^u(\sigma))$ has dense unstable branches (c.f. Lemma 4.1 in \cite{mp1}).
So, $Cl(W^u(\sigma))$ is a sectional-hyperbolic attractor by Proposition \ref{thD}
applied to $\Lambda=Cl(W^u(\sigma))$. Since $\sigma\in Cl(W^u(\sigma))$ we obtain the result.
\end{proof}

The last ingredient is the proposition below whose
proof follows from Theorem B of \cite{gw} as in the proof of Theorem B p. 1582 of \cite{mp1}.

\begin{proposition}
\label{star=>sec-axa}
If $n\geq 3$, every $C^1$-generic star flow whose
singularities accumulated by periodic orbits
belong to a sectional-hyperbolic attractor up to flow-reversing is sectional-Axiom A.
\end{proposition}

\begin{proof}[Proof of Theorem A]
Let $X$ be a $C^1$-generic vector field on a closed $n$-manifold, $n\geq 3$,
all of whose singularities accumulated by periodic orbits
have codimension one.
Suppose in addition that there is no point accumulated by hyperbolic periodic orbits of different Morse indices.
Since $X$ is $C^1$-generic we have by Proposition \ref{p1} that $X$ is a star flow.

If $n=3$ then, since $X$ is generic, Theorem B in \cite{mp1} implies that $X$ is sectional-Axiom A.

If $n\geq 4$ then, by Proposition \ref{p1}, since the singularities accumulated
by periodic orbits have codimension one,
we have that all such singularities belong to a sectional-hyperbolic attractor
up to flow-reversing.
Then, $X$ is sectional-Axiom A by Proposition \ref{star=>sec-axa}.
\end{proof}

Now we move to the proof of Theorem B.

Hereafter we denote by $W^s_X(\cdot)$ and $W^u_X(\cdot)$ the
stable and unstable manifold operations \cite{hps} with emphasis on $X$. Notation $O(p)$ (or $O_X(p)$
to emphasize $X$) will indicate the orbit of $p$ with respect to $X$.
By a {\em periodic point} we mean a point belonging to a periodic orbit of $X$.
As usual the notation $\pitchfork$ will indicate the transversal intersection operation.

\begin{lemma}
\label{l1}
There exists a residual subset $\SR\subset \mathcal{X}^1$ with the following property:
If $X\in \SR$ has two periodic points $q_0$ and $p_0$ such that for any
neighborhood $\SU$ of $X$ there exists $Y\in \SU$ such that the continuations
of $q(Y)$ and $p(Y)$ of $q_0$ and $p_0$ respectively are defined and satisfy
$W^s_Y(O(q(Y)))\pitchfork W^u_Y(O(p(Y)))\neq \emptyset$. Then $X$ satisfies
$$W^s_X(O(q_0))\pitchfork W^u_X(O(p_0))\neq \emptyset.$$
\end{lemma}

\begin{proof}
Indeed, let $\{U_n\}$ be a countable basis of the topology of $M$. Now,
we define the set $A_{n,m}$ as the set of vector fields such that there
exist a periodic point $p$ in $U_n$ and a periodic point $q$ in $U_m$
such that $W^s(O(p))\pitchfork W^u(O(q))\neq \emptyset$. Observe that
$A_{n,m}$ is an open set.
Define $B_{n,m}=\mathcal{X}^1\setminus Cl(A_{n,m})$. Thus the set
$$\SR=\bigcap_{n,m=0}^{\infty} (A_{n,m}\cup B_{n,m})$$
is residual.
If $X$ belongs to $\SR$ and satisfies the hypothesis then there
exist $n$ and $m$ such that $p_0\in U_n$ and $q_0\in U_m$. Moreover,
the hypothesis implies that $X\notin B_{n,m}$.
Thus $X\in A_{n,m}$ and the proof follows.
\end{proof}

We use this lemma to prove the following one.

\begin{lemma}
\label{l3}
A $C^1$ generic star flow with spectral decomposition
has no points accumulated by hyperbolic periodic orbits of different Morse indices.
\end{lemma}

\begin{proof}
Let $\SR$ be the residual subset in Lemma \ref{l1}.
Suppose that $X\in \SR$ has spectral decomposition
but has no points accumulated by hyperbolic periodic orbits of different Morse indices.
Then, there exists $i\neq j$ such that
$Cl(Per_i(X))\cap Cl(Per_j(X))\neq \emptyset$. Without loss of generality we can assume $i<j$.
Take $x\in Cl(Per_i(X))\cap Cl(Per_j(X))$
so there are periodic orbits $O(p_0)$ (of index $i$) and $O(q_0)$ (of index $j$) arbitrarily close to $x$.
Clearly $x\in\Omega(X)$ and so there is a basic set $\Lambda$ in the spectral
decomposition of $X$ such that $x\in \Lambda$.
As the basic sets in the spectral decomposition are disjoint and
the orbits $O(p_0)$, $O(q_0)$ are close to $x$ (and belong to $\Omega(X)$)
we conclude that $O(p_0)\cup O(q_0)\subset \Lambda$.

Since $\Lambda$ is transitive, the connecting lemma \cite{h} implies that there exists $Y$
arbitrarily close to $X$ such that
$W^s_Y(O(q(Y))) \cap W^u_Y(O(p(Y)))\neq\emptyset$.
On the other hand, $j-i>0$ since $j>i$.
Moreover, $\di (W^s_Y(O(q(Y))))=j+1$ and $\di (W^u_Y(O(p(Y))))=n-i$
since $ind(O(q(y)))=j$ and $ind(O(p(Y)))=i$ (resp.).
Then, $\di (W^s_Y(O(q(Y))))+\di (W^u_Y(O(p(Y))))=j+1+n-i>n$ and so
with another perturbation we can assume that the above intersection is transversal.
Since $X\in \SR$ we conclude that
$$W^s_X(O(q_0))\pitchfork W^u_X(O(p_0))\neq \emptyset.$$
Now, using the connecting lemma again, there exists $Y$ close to $X$ with a heterodimensional cycle.
But this contradicts the non-existence of heteroclinic cycles for star flows
(c.f. Theorem 4.1 in \cite{gw}). The proof follows.
\end{proof}

\begin{proof}[Proof of Theorem B]
Apply Theorem A and Lemma \ref{l3}.
\end{proof}

\end{document}